\documentclass[abstracton]{scrartcl}

\usepackage[latin1]{inputenc}
\usepackage{amsfonts}
\usepackage{amsmath}
\usepackage{amssymb}
\usepackage{amsthm}

\usepackage{url}

\usepackage[pdftex,colorlinks,breaklinks,linkcolor=black,citecolor=black,filecolor=black,menucolor=black,urlcolor=black,pdfauthor={Lizhen Ji, Andreas Weber},pdftitle={Pointwise bounds for L^2 eigenfunctions on locally symmetric spaces}, plainpages=false,pdfpagelabels,bookmarksnumbered=true]{hyperref}

\newtheorem{theorem}{Theorem}[section]
\newtheorem{lemma}[theorem]{Lemma}
\newtheorem{proposition}[theorem]{Proposition}
\newtheorem{corollary}[theorem]{Corollary}

\theoremstyle{definition}
\newtheorem{definition}[theorem]{Definition}

\theoremstyle{remark}
\newtheorem{remark}[theorem]{\bf Remark}

\newcommand{\N}{\mathbb{N}}
\newcommand{\Q}{\mathbb{Q}}
\newcommand{\R}{\mathbb{R}}
\newcommand{\Z}{\mathbb{Z}}
\newcommand{\C}{\mathbb{C}}
\newcommand{\K}{\mathbb{K}}

\newcommand*\e{\mathrm{e}}

\newcommand*\re{\mathrm{Re}}

\newcommand*\isom{\mathrm{Isom}}
\newcommand*\vol{\mathrm{vol\,}}
\newcommand*\supp{\mathrm{supp}}

\newcommand*\rank{\mathrm{rank}}
\newcommand*\qrank{\Q\mbox{-}\mathrm{rank}}
\newcommand*\rrank{\R\mbox{-}\mathrm{rank}}

\newcommand*\tr{\mathrm{tr}}

\newcommand*\diam{\mathrm{diam}}

\newcommand\ad{\mathrm{ad}}
\newcommand\Ad{\mathrm{Ad}}

\newcommand*\glnc{\mathrm{\it GL}(n,\C)}

\newcommand\bG{{\bf G}}
\newcommand\bT{{\bf T}}
\newcommand\bP{{\bf P}}
\newcommand\bN{{\bf N}}
\newcommand\bL{{\bf L}}
\newcommand\bS{{\bf S}}
\newcommand\bM{{\bf M}}

\newcommand{\Si}{\mathcal{S}}

\newcommand*\DMp{\Delta_{M,p}}                 

\newcommand*\DM{\Delta_M}
\newcommand*\DX{\Delta_X}

\title{Pointwise bounds for $\boldsymbol{L^2}$ eigenfunctions on locally symmetric spaces}

\author{Lizhen Ji\footnote{Email: lji@umich.edu, 
					Address: 1834 East Hall, Ann Arbor, MI 48109-1043, USA.
					 Partially supported by NSF grant DMS 0604878.}\\ 
					{\large Department of Mathematics, University of Michigan}
\and Andreas Weber\footnote{    Email: andreas.weber@math.uni-karlsruhe.de,
						   Address:  Englerstr. 2, 76128 Karlsruhe, Germany.}\\
						   {\large  Institut f\"ur Algebra und Geometrie,
						   Universit\"at Karlsruhe (TH)} }

\date{}

\begin{document}

\maketitle 

\begin{abstract} We prove pointwise bounds for $L^2$ eigenfunctions of the Laplace-Beltrami operator
on locally symmetric spaces with $\Q$-rank one if the corresponding eigenvalues lie below the continuous part of the $L^2$ spectrum. Furthermore, we use these bounds in order to obtain some results concerning the $L^p$ spectrum.\\

\end{abstract}

\section{Introduction}\label{introduction}

Let $M=\Gamma\backslash X$ denote a complete locally symmetric space such that its
universal covering $X=G/K$ is a symmetric space of non-compact type and 
$\Gamma\subset G$ is a non-uniform arithmetic lattice. Then, one knows that
the $L^2$ spectrum $\sigma(\DM)$ of the Laplace-Beltrami operator $\DM$ on the (non-compact)
locally symmetric space $M$ is the union of a point spectrum and an absolutely continuous spectrum.
The point spectrum consists of a (possibly infinite) sequence of eigenvalues
$$0=\lambda_0 < \lambda_1\leq \lambda_2\leq \dots$$
with finite multiplicities such that below any finite number there are only finitely many eigenvalues.
The absolutely continuous spectrum equals $[b,\infty)$ for some constant $b>0$ that can be described
in terms of the rational roots with respect to the semi-simple Lie group $G$. In the case where
the $\Q$ rank of $\Gamma$ equals one, we have $b=||\rho_{\bP}||^2$ (for a Definition of
$\rho_{\bP}$ see Section \ref{Q roots}).

Since $\DM$ is a self-adjoint operator on the Hilbert space $L^2(M)$ and the corresponding
heat semigroup $\e^{-t\DM}$ on $L^2(M)$ is positive and a contraction on $L^{\infty}(M)$
for any $t\geq 0$, the semigroup $\e^{-t\DM}$ leaves the set $L^1(M)\cap L^{\infty}(M)\subset L^2(M)$
invariant and hence, $\e^{-t\DM}\big|_{L^1\cap L^{\infty}}$ may be extended to a positive
contraction semigroup $T_p(t)$ on $L^p(M)$ for $p\in [1,\infty]$, and these semigroups are 
strongly continuous if $p\in [1,\infty)$. For a proof of this see  \cite[Theorem 1.4.1]{MR1103113}. 
If we denote by $-\DMp$ the generator of $T_p(t)$, the $L^p$ spectrum of $M$ is by definition
the spectrum $\sigma(\DMp)$ of $\DMp$. Note, that we have $\DM= \Delta_{M,2}$. In general,
the spectra $\sigma(\DMp)$ may depend on $p$ in a non-trivial manner. Whether this happens
is related to the volume growth of the respective Riemannian manifold $M$, see e.g.  \cite{MR1250269}. In our case, where $M=\Gamma\backslash X$ is a non-compact locally symmetric space, if we assume also that the $\Q$-rank
of $\Gamma$ equals one, it follows from \cite{Weber:2006fk,MR2342629} that the $L^p$ spectrum
depends on $p$. For similar results see \cite{MR937635,Weber:2007fk2}.\\
In contrast to the $L^2$ case where the spectrum $\sigma(\DM)$ is completely determined,
the $L^p$ spectrum $\sigma(\DMp)$ is not known in general (even as {\em set}). In the following,
we discuss some results in this direction. \\
From \cite[Proposition 3.3]{MR1016445} we may conclude that
$$ \sigma(\DMp) \subset \{\lambda_0,\ldots,\lambda_r\}\cup P_p,$$
where $P_p$ is the parabolic region
\begin{equation}
 P_p = \left\{ ||\rho_{\bP}||^2 - z^2 : z\in\C, |\re z|\leq ||\rho||\cdot |\frac{2}{p}-1|\right\}\subset \C.
\end{equation} 
For a definition of $\rho$ see Section \ref{preliminaries}. Furthermore,
from the the arguments in \cite{MR937635} it follows that an $L^2$ eigenfunction for
the eigenvalue $\lambda_j$ belongs to $L^p(M)$  and hence is also an eigenfunction for $\DMp$ whenever $\lambda_j$ is not contained in $P_p$. On the other hand, it is shown in
\cite{Weber:2006fk,MR2342629} that if the $\Q$-rank of $\Gamma$ equals one, we have
\begin{equation}\label{Lp spectrum 2}
 \{\lambda_0,\ldots,\lambda_r\}\cup P_p' \subset \sigma(\DMp),
\end{equation}
where
\begin{equation}
 P_p' =  \left\{ ||\rho_{\bP}||^2 - z^2 : z\in\C, |\re z|\leq ||\rho_{\bP}||\cdot |\frac{2}{p}-1|\right\}\subset P_p.
\end{equation}
If $X$ is a rank one symmetric space, it happens that $P_p=P_p'$ but in the higher rank case,
the inclusion is in general strict.
It is also conjectured that $P_p'$ is the ``right'' parabolic region, i.e. that equality holds in
(\ref{Lp spectrum 2}). Our results in Section \ref{Lp theory for eigenfunctions} support this,
cf. the discussion below and after Theorem \ref{Lp eigenfunction}.

The main purpose of this paper is to determine pointwise bounds for the $L^2$ eigenfunctions
of $\DM$ that correspond to the eigenvalues below the continuous spectrum $[b,\infty)$ if the 
$\Q$-rank of $\Gamma$ equals one.\\
We apply these bounds in Section \ref{Lp theory for eigenfunctions}
to show that each $L^2$ eigenfunction with eigenvalue $\lambda < b$ lies in $L^p(M)$
if $p$ is contained in some interval $\big[1,p(\lambda)\big)$ where $p(\lambda)>2$ is in general
larger than the number that can be derived from the results about the $L^p$ spectrum in \cite{MR937635,MR1016445} and is exactly the number that could be derived from the results
in  \cite{MR937635,MR1016445} if one knew that the $L^p$ spectrum coincides with
$\{\lambda_0,\ldots,\lambda_r\}\cup P_p'$.

Our methods are inspired by a similar result in the hyperbolic case (cf. \cite{MR937635}). 
We first use wave equation techniques as in \cite{MR658471}, in particular finite propagation speed,
in order to derive an upper bound for the integral 
$$\int_{S(R)} |f|^2 dvol_M,$$
where $\DM f=\lambda f$ for some $\lambda <b$ and $S(R)=\{ p\in M : d(p,p_0)>R\}$ for 
some $p_0\in M$. The obtained bound in turn, together with Sobolev embedding, elliptic regularity, and
some geometric considerations lead to a pointwise bound for $f$ (see Theorem \ref{pointwise bound}).

We want to emphasize that these methods are more general than the methods from the theory of 
automorphic forms in the sense that they can be applied to more general situations in which  the 
Fourier expansion of  eigenfunctions in the cusps is no longer available. This happens for example, in case of certain perturbations of the locally symmetric metric that destroy the local symmetry in the cusps, cf. Remark \ref{remark}.

\section{Symmetric spaces}\label{preliminaries}\label{symmetric spaces}

Let $X$ denote always a symmetric space of non-compact type. Then
$G:= \isom^0(X)$ is a non-compact, semi-simple Lie group with trivial center 
that acts transitively on $X$ and $X=G/K$, where $K\subset G$ is a maximal 
compact subgroup of $G$. We denote
the respective Lie algebras by $\mathfrak{g}$ and $\mathfrak{k}$. Given a corresponding Cartan
involution $\theta: \mathfrak{g}\to\mathfrak{g}$ we obtain the Cartan decomposition
$\mathfrak{g}=\mathfrak{k}\oplus\mathfrak{p}$ of $\mathfrak{g}$ into the eigenspaces of $\theta$. The subspace
$\mathfrak{p}$ of $\mathfrak{g}$ can be identified with the tangent space $T_{eK}X$. We assume,
that the Riemannian metric $\langle\cdot,\cdot\rangle$ of $X$ in $\mathfrak{p}\cong T_{eK}X$ 
coincides with the restriction of the Killing form 
$B(Y,Z) := \tr(\ad Y\circ \ad Z ), Y, Z\in \mathfrak{g},$ to $\mathfrak{p}$. 

For any maximal abelian subspace $\mathfrak{a}\subset \mathfrak{p}$ we refer to 
$\Sigma=\Sigma(\mathfrak{g},\mathfrak{a})$ as the set of restricted roots for the pair $(\mathfrak{g},\mathfrak{a})$,
i.e. $\Sigma$ contains all $\alpha\in \mathfrak{a}^*\setminus\{0\}$ such that
$$ \mathfrak{h}_{\alpha} := \{ Y\in \mathfrak{g} : \ad(H)(Y) = \alpha(H)Y \mbox{~for all~} H\in\mathfrak{a} \}\neq \{0\}.$$
These subspaces $ \mathfrak{h}_{\alpha}\neq \{0\}$ are called root spaces.\\
Once a positive Weyl chamber $\mathfrak{a}^+$ in $\mathfrak{a}$ is chosen, we denote by
$\Sigma^+$ the  subset of positive roots and by 
$\rho:= \frac{1}{2}\sum_{\alpha\in\Sigma^+} (\dim \mathfrak{h}_{\alpha})\alpha$ 
half the sum of the positive roots (counted according to their multiplicity).

\subsection{Arithmethic groups and $\Q$-rank}

Since $G= \isom^0(X)$ is a non-compact, semi-simple Lie group with trivial center,
we can find a connected, semi-simple algebraic group $\bG\subset \glnc$ defined over $\Q$ 
such that the groups $G$ and $\bG(\R)^0$ are isomorphic as Lie groups 
(cf. \cite[Proposition 1.14.6]{MR1441541}). 

Let us denote by $\bT_{\K}\subset \bG$ ($\K=\R$ or $\K=\Q$) a maximal $\K$-split algebraic
torus in $\bG$. Remember that we call a closed subgroup $\bT$ of $\bG$ a {\em torus} if
$\bT$ is diagonalizable over $\C$, or equivalently if $\bT$ is abelian and every element of
$\bT$ is semi-simple. Such a torus $\bT$ is called $\R$-split if  $\bT$ is diagonalizable over
$\R$ and $\Q$-split if $\bT$ is defined over $\Q$ and diagonalizable over $\Q$.\\
All maximal $\K$-split tori in $\bG$ are conjugate under $\bG(\K)$, and we call their common dimension
$\K$-{\em rank} of $\bG$. It turns out that the  $\R$-$\rank$ of $\bG$ coincides with the rank of the symmetric space $X=G/K$, i.e. the dimension of a maximal flat subspace in $X$. 

Since we are only interested in
{\em non-uniform} lattices $\Gamma\subset G$ (i.e. finite co-volume but not co-compact), we may define 
arithmetic lattices in the following way (cf.  \cite[Corollary 6.1.10]{MR776417} and its proof):

\begin{definition}\label{definition of arithmetic subgroup}
 A non-uniform lattice $\Gamma\subset G$ in a connected semi-simple Lie group $G$ with trivial center 
 and no compact factors is called {\em arithmetic} if there are
 \begin{itemize}
   \item[\textup{(i)}] a  semi-simple algebraic group $\bG\subset \glnc$ defined over $\Q$ and 
   \item[\textup{(ii)}] an isomorphism 
   			$$\varphi: \bG(\R)^0\to G$$
 \end{itemize}
such that $\varphi(\bG(\Z)\cap \bG(\R)^0)$ and\, $\Gamma$ are commensurable, i.e. 
$\varphi(\bG(\Z)\cap \bG(\R)^0)\cap \Gamma$ has finite index in both $\varphi(\bG(\Z)\cap \bG(\R)^0)$ and $\Gamma$.
\end{definition}
\noindent For the general definition of arithmetic lattices see \cite[Definition 6.1.1]{MR776417}.

A well-known and fundamental result due to Margulis ensures that this is usually the only
way to obtain a lattice. More precisely, every irreducible lattice $\Gamma\subset G$
in a   connected, semi-simple Lie group $G$ with trivial center, no compact
factors and $\rrank(G) \geq 2$ is arithmetic (\cite{MR1090825,MR776417}).

Further results due to  Corlette (cf.  \cite{MR1147961})  and  Gromov \&  Schoen (cf. \cite{MR1215595}) extended this result to all connected semi-simple Lie groups
with trivial center except $SO(1,n)$ and $SU(1,n)$. In $SO(1,n)$ (for all $n\in \N$) and
in $SU(1,n)$ (for $n= 2,3$)  actually non-arithmetic
lattices are known to exist (see e.g. \cite{MR932135,MR1090825}).

\begin{definition}{\bf ($\Q$-rank of an arithmetic lattice).}
 Suppose $\Gamma\subset G$ is an arithmetic lattice in a connected semi-simple 
 Lie group $G$ with trivial center and no compact factors.  Then 
 $\Q$-$\rank(\Gamma)$ is by definition the $\Q$-$\rank$ of $\bG$, where $\bG$ is 
 an algebraic group as in Definition \ref{definition of arithmetic subgroup}.
\end{definition}
The theory of algebraic groups shows that the definition of the $\Q$-rank of an arithmetic lattice
 does not depend on the choice of the algebraic group $\bG$ in Definition \ref{definition of arithmetic subgroup}.  A proof of this fact can be found in  \cite[Corollary 9.12]{math.DG/0106063}.

\subsection{Siegel sets and reduction theory}\label{Siegel Sets and Reduction Theory}

Let us denote in this subsection by $\bG$ again a connected, semi-simple algebraic 
group defined over $\Q$ with trivial center and by $X=G/K$ the corresponding symmetric space 
of non-compact type with $G=\bG^0(\R)$. Our main references in this subsection are \cite{MR0244260,MR2189882,MR1906482}.
\subsubsection{Langlands decomposition of rational parabolic subgroups}
\begin{definition}
A closed subgroup $\bP\subset \bG$ defined over $\Q$ is called {\em rational parabolic subgroup}
if  $\bP$ contains a maximal, connected solvable subgroup of $\bG$. (These subgroups are
    called {\em Borel subgroups} of $\bG$.)
\end{definition}

For any rational parabolic subgroup $\bP$ of $\bG$ we denote by $\bN_{\bP}$ the unipotent
radical of $\bP$, i.e. the largest unipotent normal subgroup of $\bP$ and by 
$N_{\bP}:= \bN_{\bP}(\R)$ the real points of $\bN_{\bP}$.
The {\em Levi quotient} $\bL_{\bP}:= \bP/\bN_{\bP}$ is reductive and both $\bN_{\bP}$ and
$\bL_{\bP}$ are defined over $\Q$. If we denote by $\bS_{\bP}$ the maximal $\Q$-split torus
in the center of $\bL_{\bP}$ and by $A_{\bP}:= \bS_{\bP}(\R)^0$ the connected component
of $\bS_{\bP}(\R)$ containing the identity, we obtain the decomposition of $\bL_{\bP}(\R)$ into
$A_{\bP}$ and the real points $M_{\bP}$ of a reductive algebraic group $\bM_{\bP}$ defined over $\Q$:
$$ \bL_{\bP}(\R) = A_{\bP}M_{\bP} \cong A_{\bP}\times M_{\bP}.$$

After fixing a certain basepoint $x_0\in X$, we can lift the groups 
$ \bL_{\bP}, \bS_{\bP}$ and $\bM_{\bP}$ into $\bP$ such that their images
$ \bL_{\bP, x_0}, \bS_{\bP, x_0}$ and $\bM_{\bP, x_0}$ are algebraic groups defined over
$\Q$ and give rise to the {\em rational Langlands decomposition} of $P:=\bP(\R)$:
$$ P \cong N_{\bP}\times A_{\bP, x_0}\times M_{\bP, x_0}.$$
More precisely, this means that the map
$$  N_{\bP}\times A_{\bP, x_0}\times M_{\bP, x_0}\to P,\quad
       \left( n, a, m\right) \mapsto nam$$
is a real analytic diffeomorphism.  

Denoting by $X_{\bP, x_0}$ the {\em boundary symmetric space}
$$ X_{\bP, x_0} := M_{\bP, x_0}/ K\cap M_{\bP, x_0}$$
we obtain, since the subgroup $P$ acts transitively on the symmetric space $X=G/K$ (we actually have $G=PK$), the following {\em rational horocyclic decomposition} of $X$:
$$
 X\cong N_{\bP}\times A_{\bP, x_0}\times X_{\bP, x_0}.
$$ 
More precisely, if we denote by $\tau: M_{\bP, x_0}\to X_{\bP, x_0}$ the canonical projection, we have an analytic diffeomorphism
\begin{equation}\label{rational horocyclic decomposition}
 \mu: N_{\bP}\times A_{\bP, x_0}\times X_{\bP, x_0} \to X,\,\, (n,a,\tau(m)) \mapsto nam\cdot x_0.
\end{equation}

Note, that the boundary symmetric space $X_{\bP, x_0}$ is a Riemannian product of a symmetric
space of non-compact type by a Euclidean space. 

For minimal rational parabolic subgroups, i.e. Borel subgroups $\bP$, we have
$$ \dim A_{\bP, x_0} = \qrank(\bG).$$
In the following we omit the reference to the chosen basepoint $x_0$ in the subscripts.

\subsubsection{$\Q$-Roots}\label{Q roots}

Let us fix some {\em minimal} rational parabolic subgroup $\bP$ of $\bG$. We denote in the
following by $\mathfrak{g}, \mathfrak{a}_{\bP}$, and $\mathfrak{n}_{\bP}$ the Lie algebras of the (real) Lie groups
$G, A_{\bP}$, and $N_{\bP}$ defined above. Associated with the pair $(\mathfrak{g}, \mathfrak{a}_{\bP})$ there is
-- similar to Section \ref{symmetric spaces} --  a system $\Phi(\mathfrak{g}, \mathfrak{a}_{\bP})$ 
of  so-called {\em $\Q$-roots}. If we define for $\alpha\in \Phi(\mathfrak{g}, \mathfrak{a}_{\bP})$ the {\em root
spaces}
$$ \mathfrak{g}_{\alpha} := \{ Y\in \mathfrak{g} : \ad(H)(Y) = \alpha(H)(Y) \mbox{~for all~} H\in \mathfrak{a}_{\bP} \},$$
we have the root space decomposition
$$ \mathfrak{g} = \mathfrak{g}_0 \oplus \bigoplus_{\alpha\in \Phi(\mathfrak{g}, \mathfrak{a}_{\bP})} \mathfrak{g}_{\alpha},$$
where $\mathfrak{g}_0$ is the Lie algebra of $Z(\bS_{\bP}(\R))$, the center of $\bS_{\bP}(\R)$. 
Furthermore, the minimal rational 
parabolic subgroup $\bP$ defines an ordering of $\Phi(\mathfrak{g}, \mathfrak{a}_{\bP})$ such that
$$ \mathfrak{n}_{\bP} = \bigoplus_{\alpha\in \Phi^+(\mathfrak{g}, \mathfrak{a}_{\bP})} \mathfrak{g}_{\alpha}.$$
The root spaces $\mathfrak{g}_{\alpha}, \mathfrak{g}_{\beta}$ to distinct  positive roots 
$\alpha, \beta\in \Phi^+(\mathfrak{g}, \mathfrak{a}_{\bP})$ are orthogonal with respect to the Killing form:
$$ B(\mathfrak{g}_{\alpha}, \mathfrak{g}_{\beta}) = \{0\}.$$
In analogy to Section  \ref{symmetric spaces} we define
$$ \rho_{\bP} := \sum_{\alpha\in\Phi^{+}(\mathfrak{g}, \mathfrak{a}_{\bP})}(\dim\mathfrak{g}_{\alpha})\alpha.$$   
Furthermore, we denote by $\Phi^{++}(\mathfrak{g}, \mathfrak{a}_{\bP})$ the set of simple positive
roots. Recall, that we call a positive root $\alpha\in \Phi^{+}(\mathfrak{g}, \mathfrak{a}_{\bP})$ simple if
$\frac{1}{2}\alpha$ is not a root.\\

\begin{remark}          
The elements of $\Phi(\mathfrak{g}, \mathfrak{a}_{\bP})$ are differentials of characters of the maximal 
$\Q$-split torus $\bS_{\bP}$.  For convenience, we identify the $\Q$-roots with characters. 
If restricted to $A_{\bP}$ we denote therefore the values of these characters by 
$\alpha(a), (a\in A_{\bP}, \alpha \in \Phi(\mathfrak{g}, \mathfrak{a}_{\bP}) )$ which is defined by
$$ \alpha(a) := \exp\alpha(\log a).$$
\end{remark}

\subsubsection{Siegel sets}

Since we will consider in the succeeding section only (non-uniform) arithmetic lattices $\Gamma$ 
with $\qrank(\Gamma)=1$, we restrict ourselves from now on to the case
$$ \qrank(\bG) =1.$$
For these groups we summarize  several facts in the next lemma.
\begin{lemma}
Assume $\qrank(\bG) =1$. Then the following holds:
 \begin{itemize}
  \item[\textup{(1)}] For any proper rational parabolic subgroup $\bP$ of $\bG$, 
  	we have $\dim A_{\bP}=1$.
  \item[\textup{(2)}] All proper rational parabolic subgroups are minimal.
  \item[\textup{(3)}] The set $\Phi^{++}(\mathfrak{g}, \mathfrak{a}_{\bP})$ of simple positive $\Q$-roots 
      contains only a single element:
    $$\Phi^{++}(\mathfrak{g}, \mathfrak{a}_{\bP}) =\{\alpha\}.$$
 \end{itemize}
\end{lemma}

For any rational parabolic subgroup $\bP$ of $\bG$ and any $t > 1$, we define
$$ A_{\bP,t} := \{ a\in A_{\bP} : \alpha(a) > t \},$$
where $\alpha$ denotes the unique root in $\Phi^{++}(\mathfrak{g}, \mathfrak{a}_{\bP})$.\\
If we choose $a_0\in A_{\bP}$ with the property $\alpha(a_0)= t$, the set $A_{\bP,t}$
is just a shift of the positive Weyl chamber $A_{\bP,1}$ by $a_0$:
$$ A_{\bP,t} = A_{\bP,1}a_0.$$

Before we define {\em Siegel sets}, we recall the rational horocyclic decomposition of the 
symmetric space $X=G/K$:
$$ X\cong N_{\bP}\times A_{\bP}\times X_{\bP}.$$

\begin{definition}
  Let $\bP$ denote a rational parabolic subgroup of the algebraic group
  $\bG$ with $\qrank(\bG)=1$. For any bounded set $\omega\subset N_{\bP}\times X_{\bP}$
  and any $t >1$, the set
  $$ \Si_{\bP, \omega, t} := \omega\times A_{\bP, t}\subset X$$
  is called {\em Siegel set}.
\end{definition} 

\subsubsection{Precise reduction theory}

We fix an arithmetic lattice $\Gamma\subset G=\bG(\R)$ in the algebraic group
$\bG$ with $\qrank(\bG)=1$.
Recall, that by a well known result due to A. Borel and Harish-Chandra 
there are only finitely many $\Gamma$-conjugacy classes of minimal  parabolic subgroups
(see e.g. \cite{MR0244260}).
Using the Siegel sets defined above, we can state the {\em precise reduction theory}
in the $\qrank$ one case as follows:

\begin{theorem}\label{precise reduction theory}
 Let $\bG$ denote a semi-simple algebraic group defined over $\Q$ with\break
  $\qrank(\bG)=1$ and $\Gamma$
 an arithmetic lattice in $G$. We further denote by $\bP_1,\ldots, \bP_k$ representatives of
 the $\Gamma$-conjugacy classes of all rational proper (i.e. minimal) parabolic subgroups
 of $\bG$. Then there exist a bounded set $\Omega_0\subset X$ and Siegel sets
 $\omega_j\times A_{\bP_j, t_j}\, (j=1,\ldots, k)$ such that the following holds:
 \begin{itemize}
  \item[\textup{(1)}] Under the canonical projection $\pi: X\to \Gamma\backslash X$ each Siegel set
     $\omega_j\times A_{\bP_j, t_j}$ is mapped injectively into $\Gamma\backslash X, \,
     i=1,\ldots, k.$
  \item[\textup{(2)}] The image of $\omega_j$ in $(\Gamma\cap P_j)\backslash N_{\bP_j}
  	\times X_{\bP_j}$
   	is compact $(j=1,\ldots, k)$.
   \item[\textup{(3)}] The subset 
   	$$\Omega_0\cup\coprod_{j=1}^k \omega_j\times A_{\bP_j, t_j}$$ 
   	is an open fundamental domain for $\Gamma$. In particular, $\Gamma\backslash X$ equals
	the closure of $\pi(\Omega_0)\cup\coprod_{j=1}^k \pi(\omega_j\times A_{\bP_j, t_j}).$
 \end{itemize}
\end{theorem}

Geometrically this means that the closure of each set $\pi(\omega_j\times A_{\bP_j, t_j})$ 
corresponds to one cusp of the locally symmetric
space $\Gamma\backslash X$ and the numbers $t_j$ are chosen large enough such that
these sets do not overlap. Then the interior of the bounded set $\pi(\Omega_0)$ is just 
the complement of the closure of $\coprod_{j=1}^k \pi(\omega_j\times A_{\bP_j, t_j})$.

Since in the case $\qrank(\bG)=1$ all rational proper parabolic subgroups are minimal,
these subgroups are conjugate under $\bG(\Q)$ (cf.  \cite[Theorem 11.4]{MR0244260}).
Therefore, the root systems $\Phi(\mathfrak{g},\mathfrak{a}_{\bP_j})$ with respect to the rational
proper parabolic subgroups $\bP_j$, $j=1\ldots k$, are canonically isomorphic 
(cf.  \cite[11.9]{MR0244260}) and moreover, we can conclude 
$||\rho_{\bP_1}|| = \ldots = ||\rho_{\bP_k}||$.

\subsection{Rational horocyclic coordinates}

For all $\alpha\in\Phi^+(\mathfrak{g},\mathfrak{a}_{\bP})$ we define on 
$\mathfrak{n}_{\bP}= \bigoplus_{\alpha\in \Phi^+(\mathfrak{g},\mathfrak{a}_{\bP})}\mathfrak{g}_{\alpha}$ a left invariant
bilinear form $h_{\alpha}$ by
$$ h_{\alpha} := \left\{
	\begin{array}{ll}
	 \langle\cdot,\cdot\rangle, & \mbox{on~} \mathfrak{g}_{\alpha}\\
	 0, & \mbox{else},
	\end{array}\right.$$
where $\langle Y, Z \rangle := -B(Y,\theta Z)$ denotes the usual $\Ad(K)$-invariant bilinear
form on $\mathfrak{g}$ induced from the Killing form $B$.
We then have (cf.   \cite[Proposition 1.6]{MR0338456} or  \cite[Proposition 4.3]{MR0387496}):
\begin{proposition}
  \begin{itemize}
   \item[\textup{(a)}]
   For any $x=(n,\tau(m),a)\in X \cong N_{\bP}\times X_{\bP}\times A_{\bP}$ the tangent spaces at $x$
   to the submanifolds $\{n\}\times X_{\bP}\times\{a\},\, \{n\}\times\{\tau(m)\}\times A_{\bP}$, and
   $N_{\bP}\times\{\tau(m)\}\times\{a\}$ are mutually orthogonal. 
   \item[\textup{(b)}]
   The pullback $\mu^*g$
   of the metric $g$ on $X$ to $N_{\bP}\times X_{\bP}\times A_{\bP}$ is given by
   $$ ds^2_{(n,\tau(m),a)} = 
   	\frac{1}{2}\sum_{\alpha\in\Phi^+(\mathfrak{g},\mathfrak{a}_{\bP})}\e^{-2\alpha(\log a)}h_{\alpha}\oplus d(\tau(m))^2
	\oplus da^2.$$
  \end{itemize} 
\end{proposition}

If we choose  orthonormal bases $\{N_1,\ldots, N_r\}$
of $\mathfrak{n}_{\bP}, \{Y_1,\ldots ,Y_l\}$ of some tangent space $T_{\tau(m)}X_{\bP}$ and 
$H\in \mathfrak{a}_{\bP}^+$ with $||H||=1$, we obtain {\em rational horocyclic coordinates}
  \begin{eqnarray*}
     \varphi &:&     N_{\bP}\times X_{\bP}\times A_{\bP} 
                                  \to \R^r\times\R^l\times\R, \\
          {}      & {}&    \left(\exp(\sum_{j=1}^r x_jN_j), \exp(\sum_{j=1}^l x_{j+r}Y_j), \exp(yH)\right)
         		           \mapsto       (x_1,\ldots, x_{r+l},y).
  \end{eqnarray*}
In the following, we will abbreviate  $(x_1,\ldots, x_{r+l},y)$ as $(x,y)$.

\begin{corollary}[see Corollary 4.4 in \cite{MR0387496}]
The volume form of $X=N_{\bP}\times X_{\bP}\times A_{\bP} $ with respect to rational horocyclic coordinates is given by
\begin{eqnarray*}
dvol_X &=&   h(x)\;\e^{-2||\rho_{\bP}|| y} dxdy,
\end{eqnarray*}      
where $\log a= yH$ and $h>0$ is smooth.      
\end{corollary}

\subsection{A Lemma}

\begin{lemma} \label{k(y)}
Let $X=G/K$ denote a symmetric space of non-compact type and $\Gamma\subset G$
an arithmetic lattice with $\Q$-rank one. Furthermore, we denote by   
$F = \Omega_0\cup\coprod_{j=1}^l \omega_j\times A_{\bP_j, t_j}$ 
the fundamental domain for $\Gamma$ from Theorem \ref{precise reduction theory} and by
$ B_y(r) = \{ x\in X : d( (w_0,\e^{yH}), x) < r\}$ 
the metric ball in $X$ with center $(w_0,\e^{yH})\in \omega_j\times A_{\bP_j, t_j} $ and radius $r>0$. 
If 
$$ k(y) = \#\{ \gamma\in\Gamma :  \gamma F\cap B_y(1)\neq\emptyset \}$$
denotes the number of $\Gamma$-translates of $B_y(1)$ whose intersection with $F$ is non-empty,
we have
$$ k(y) \leq C\e^{2||\rho_{\bP}||y},$$
for some constant $C>0$ and all $y>t_j$.
\end{lemma}
\begin{proof}
Let in the following $B_y=B_y(1)$, $\omega\times A_{\bP,t}=\omega_j\times A_{\bP_j,t_j}$,
$$\Gamma'=\{ \gamma\in \Gamma : \gamma F\cap  B_y\neq \emptyset\},$$
and 
$$ r_0 = \diam(\omega)+1.$$

To prove the desired upper bound on $k(y)$, it suffices to show that there exists
a positive constant $c_0$ such that for every $\gamma\in \Gamma'$,

$$vol_X(\gamma  F\cap B_y(r_0))\geq c_0 e^{-2||\rho_{\bP}|| y}.$$

In fact, this follows from 
$$B_y(r_0) \supseteq \bigcup_{\gamma \in \Gamma'}  \gamma F\cap B_y(r_0).$$

By assumption, $\Gamma$ is an arithmetic subgroup, and $\bP$
is a rational parabolic subgroup. 
Therefore, we have
$$ \Gamma_P = \Gamma \cap P \subset N_{\bP}M_{\bP}$$
(cf. \cite[Proposition 1.2]{MR0387495}).
Furthermore, it follows from reduction theory (cf. \cite[Proposition III.2.19]{MR2189882})
if $s\gg 0$ that $\Gamma_P$ acts on the horoball $N_{\bP}\times X_{\bP}\times A_{\bP,s}$
and a fundamental domain for this action is given by
$$ F\cap N_{\bP}\times X_{\bP}\times A_{\bP,s} = \omega\times A_{\bP,s}.$$
Hence, if $s,y \gg 0$ and $B_y\subset N_{\bP}\times X_{\bP}\times A_{\bP,s}$
the ball $B_y$ meets only $\Gamma_P$-translates of $F$, and consequently
$$ \Gamma'\subset \Gamma_P\subset  N_{\bP}M_{\bP}.$$

This means that  the action of $\gamma\in \Gamma'$ on points in $F$ does not change the $y$ coordinate and  for $\gamma\in \Gamma'$ we therefore have 
$ (\gamma\omega)\times \{\e^{yH}\} \subset B_y(r_0)$ and even
$$ (\gamma\omega)\times \{\e^{tH} : t\in (y-\varepsilon, y+\varepsilon) \} \subset B_y(r_0),$$
with some $\varepsilon >0$ (independent of $y$). The volume of the set 
$ (\gamma\omega)\times \{\e^{tH} : t\in (y-\varepsilon, y+\varepsilon) \}$ is 
$c_0 \e^{-y||\rho_{\bP}||}$ for some positive constant $c_0$ (independent of $\gamma$ and $y$) and
the claim follows.
\end{proof}

\section{Eigenfunction estimates}

In this section $X=G/K$ denotes always a symmetric space of non-compact type and
$\Gamma\subset G$ a non-uniform arithmetic lattice. 

\begin{lemma}\label{compact subset}
 There is a compact subset $M_0\subset M=\Gamma\backslash X$ of co-dimension zero such that
 $$ \langle \Delta_{M} f, f\rangle \geq || \rho_{\bP} ||^2 \cdot ||f||_{L^2}$$
 for all $f\in C_c^{\infty}(M\setminus M_0)$.
\end{lemma}
For a proof of this lemma we refer to \cite[Lemma 10]{MR937635}.

\begin{proposition}\label{s(r)}
 Let  $M=\Gamma\backslash X$, $f$  denote an eigenfunction of $\DM$ with respect 
 to some eigenvalue $\lambda < ||\rho_{\bP}||^2$,  $p_0\in M$ some arbitrary point in $M$, 
 and $S(R) := \{ p\in M : d(p,p_0) >R\}$. Then there is a constant $C>0$ such that
 \begin{equation}
  \int_{S(R)} |f|^2 dvol_M \leq C \exp\left(- 2R \sqrt{ ||\rho_{\bP}||^2 - \lambda}\right).
 \end{equation}
\end{proposition}

\begin{proof}
We choose $M_0\subset M$ as in Lemma \ref{compact subset} and a function 
$\psi \in C_c^{\infty}(M)$ with $\psi \equiv 1$ on a neighborhood of $M_0$. We further
put 
$$h := (1-\psi)f$$
and denote by $\Delta_{M\setminus M_0}$ the $L^2$ Laplace-Beltrami operator on
$M\setminus M_0$ with Dirichlet boundary conditions. Then the function $h$
is contained in the domain of $\Delta_{M\setminus M_0}$ and we have
$$ \varphi := (\Delta_{M\setminus M_0} - \lambda)h \in C_c^{\infty}(M\setminus M_0).$$
From Lemma \ref{compact subset} it follows that 
$\sigma(\Delta_{M\setminus M_0}) \subset [ ||\rho_{\bP}||^2, \infty)$
and we may define by the spectral theorem the operator
$$ A:= \Big(\Delta_{M\setminus M_0} - ||\rho_{\bP}||^2\Big)^{1/2}$$
on $L^2(M\setminus M_0)$.
If we make the definition
$$ l= \sqrt{ ||\rho_{\bP}||^2 - \lambda},$$
we obtain
\begin{eqnarray*}
 h  &=& (\Delta_{M\setminus M_0} - \lambda)^{-1}\varphi\\
 	&=& ( A^2 + l^2)^{-1}\varphi\\
	&=& \frac{1}{2l}\int_{-\infty}^{\infty} \e^{-l |t|} \cos(tA)\varphi\, dt,
\end{eqnarray*}
where the last step follows from the spectral theorem.

We now choose $a>0$ large enough such that 
$\supp \varphi$ and $M_0$ are contained in  $B(p_0, a)$. By finite propagation speed
of $\cos(tA)$ (cf. \cite{MR658471}) we may conclude that 
$$ \supp \cos(tA)\varphi \subset B(p_0, a + |t|).$$
For all $R> a$ we therefore have
$$ \left.h\right|_{S(R)} = \frac{1}{2l} \int_{ |t| > R-a} \e^{-l |t|} \left(\cos(tA)\varphi\right)\big|_{S(R)} dt.$$
Hence, by $||\cos(tA)||_{L^2\to L^2}\leq 1$, we obtain
\begin{eqnarray*}
 || h ||_{L^2(S(R))}  &\leq& \frac{1}{2l} \int_{|t| > R-a} \e^{-l |t|} || \varphi ||_{L^2(M\setminus M_0)} dt\\
 	&=& \frac{\e^{la} || \varphi ||_{L^2(M\setminus M_0)} }{l^2}\cdot \e^{-lR}\\
	&=& C\cdot\exp\left(-R  \sqrt{ ||\rho_{\bP}||^2 - \lambda}\right).
\end{eqnarray*}
As $h\big|_{S(R)} = f\big|_{S(R)}$ for large enough $R$ (the function
$\psi$ has compact support!), the result follows.
\end{proof}

\begin{theorem}\label{pointwise bound}
 Let $M= \Gamma\backslash X$ denote a locally symmetric space with $\qrank(\Gamma)=1,$
 $f$ an eigenfunction of $\DM$ with respect to some eigenvalue $\lambda < ||\rho_{\bP}||^2$,
 and $p_0\in M$ some arbitrary point in $M$. Then there is a constant $C>0$ such that
 \begin{equation}
  | f(p)| \leq C\exp\left\{ \left(||\rho_{\bP}|| - \sqrt{||\rho_{\bP}||^2-\lambda}\right)d(p,p_0)\right\}
 \end{equation}
 for any $p\in M$.
\end{theorem}

\begin{proof}
 We consider $f$ as a $\Gamma$-invariant function on the symmetric space $X$
 of non-compact type. A fundamental domain $F\subset X$ for $\Gamma$ is given
 by the (disjoint) union of finitely many Siegel sets:
 $$ F = \Omega_0\cup\coprod_{j=1}^l \omega_j\times A_{\bP_j, t_j}.$$
  As $f$ is smooth by elliptic regularity, it suffices to estimate $f$ in the unbounded Siegel sets. Let us consider in the following the Siegel set $\omega\times A_{\bP,t}$.\\
  
 We choose a point $(w_0,\e^{yH})\in\omega\times A_{\bP,t}$ and denote as above by 
 $$ B_y(r) := \{ x\in X : d( (w_0,\e^{yH}), x) < r\}$$
 the metric ball in $X$ with center $(w_0,\e^{yH})$ and radius $r$, furthermore we put
 $B_y=B_y(1)$.

 \paragraph{Geodesic normal coordinates}
 
 We now consider geodesic normal coordinates 
 $$ \psi : X\to \R^n, x\mapsto (u_1,\ldots,u_n)$$
 of $X$ with respect to the point $(w_0,\e^{yH})\in X$ for some $y\geq t$.
 Note that these coordinates may be defined globally as $X$ is simply connected and non-positively
 curved.
 
 The representation $(\tilde{g}_{ij})$ of the metric $g$ with respect to $\psi$ has 
 the following Taylor expansion around $0\in\R^n$:
 $$
  \tilde{g}_{ij}(u) = \delta_{ij} + \frac{1}{3}\sum_{k,l}R_{ikjl}(w_0,\e^{yH})u_ku_l + O(||u||^3).
 $$
 All the coefficients in this Taylor expansion are polynomials in the components of the
 Riemannian curvature tensor and its covariant derivatives (all evaluated at the point
  $(w_0,\e^{yH})$). See e.g. \cite[p. 41, Proposition 3.1]{MR1390760}.
  
 As the curvature of $X$ is bounded and $\nabla R=0$ ($X$ is a symmetric space!) we
 may conclude that the functions
 $\tilde{g}_{ij}(u)$ are bounded on $\psi(B_y)$. Furthermore, the bounds can
 be chosen such that they do {\em not} depend on $y\geq t$.
 
 From the representation
 $$
 \DX = \frac{1}{ \sqrt{\det(\tilde{g}_{ij})}} 
 	\sum_{l,k} \partial_l\Big(\tilde{g}^{lk}\sqrt{\det(\tilde{g}_{ij})} \partial_k  \Big)
 $$
 it then follows that in geodesic normal coordinates the Laplace-Beltrami operator $\DX$
 on $\psi(B_y)$ is a uniformly elliptic operator whose ellipticity bounds
 do {\em not} depend on $y\geq t$.
  
 \paragraph{Sobolev embedding}(cf. \cite[Section 7.7]{MR737190}).
 We use geodesic normal coordinates $\psi$.\\
 For  $k > \frac{n}{2},\, n=\dim X$, and some $r>0$ we have
 $$ W^{k,2}( \psi(B_y(r)) \hookrightarrow
 	C( \psi(B_y(r))),$$
 more precisely, there is a constant $C>0$ (only depending on $n$) such that
 \begin{equation}
  \sup_{ \psi(B_y(r))}|f\circ\psi^{-1}(u)| \leq
  C | \psi(B_y(r))|^l\cdot 
  	||f\circ\psi^{-1}||_{W^{k,2}( \psi(B_y(r)))},
 \end{equation}
 where $l>0$ depends only on $n$ and $| \psi(B_y(r))|$
 denotes the Lebesgue measure of the set $ \psi(B_y(r))\subset\R^n$.
 Note also, that the $W^{k,2}$ norm is built with respect to Lebesgue measure.
 
  By construction and by the representation $(\tilde{g}_{ij})$ of the metric $g$ on $X$ there
 is a constant $C_1>0$ (not depending on $y$) such that
 $$ | \psi(B_y(r))| \leq C_1 \vol_X (B_y(r)).$$
 The right hand side of this inequality does not depend on $y$ as $X$ is a homogeneous space.
 Therefore, we may find a constant $C>0$ (only depending on $n$) such that
 \begin{equation}
   \sup_{  \psi(B_y(r))}|f\circ\psi^{-1}(u)| \leq
   C ||f\circ\psi^{-1}||_{W^{k,2}(  \psi(B_y(r)) )}
 \end{equation}
 for $k> \frac{n}{2}$.
 \paragraph{Elliptic regularity}(cf. \cite[Section 8.3]{MR737190}).
 As $\DX$ is uniformly elliptic with respect to normal coordinates and as the ellipticity bounds
 on $B_y(r')$  do not depend on $y$ (for some $r'>0$ ), we may find  by elliptic regularity  a constant  $C_1>0$ only depending on $n$ such that  the following inequality holds:
 \begin{equation}
  ||f\circ\psi^{-1}||_{W^{k,2}(  \psi(B_y(r)) )} \leq
  C_1 \cdot ||f\circ\psi^{-1}||_{L^2(  \psi(B_y(r'))},
 \end{equation}
where $r'>r$ and the $L^2(  \psi(B_y(r')))$ norm is built with respect to Lebesgue measure. \\

 As the coefficients $\tilde{g}_{ij}$ of the metric $g$ in normal coordinates
 are (independently of $y\geq t$) bounded on 
 $ \psi(B_y(r))$ there is a constant $C$ (again independent of
 $y$) such that
 $$   ||f\circ\psi^{-1}||_{L^2(  \psi(B_y))} \leq
 	C || f ||_{L^2( B_y)} $$
where the norm $|| \cdot ||_{L^2(B_y)}$ is built with repect to the
Riemannian metric $g$ on $X$. 

Putting all inequalities together, we obtain
\begin{equation}
    \sup_{B_y(r)}|f(x)| \leq  C || f ||_{L^2(B_y)}
\end{equation}
with some constant $C>0$ that does only depend on $n$ and $0<r<1$ (and is in particular independent of $y$).\\

Now, if the ball $B_y$ intersects $k(y)$ different $\Gamma$ translates of $F$, i.e.
\begin{eqnarray*}
 k(y) &=& \#\{ \gamma\in\Gamma : \gamma F\cap B_y\neq\emptyset \} 
\end{eqnarray*} 
it follows from Propostion \ref{s(r)} and the inequality just derived 
$$ \sup_{B_y(r)}|f(x)| \leq  C\cdot \sqrt{k(y)}\cdot  \exp\left(- y \sqrt{ ||\rho_{\bP}||^2 - \lambda}\right).$$
The claim now follows from Lemma \ref{k(y)}.
 \end{proof}

By the decomposition principle (cf. \cite{MR544241})  compact perturbations of the Riemannian metric of $M$ leave the continuous $L^2$ spectrum (if considered as a subset of $\R$) invariant. Therefore, an
analysis of the proof of Theorem \ref{pointwise bound} shows that the preceeding result admits the following generalization.

\begin{theorem}\label{generalization}
Let $M$ be as in Theorem \ref{pointwise bound} and $h$ a compact perturbation of the
respective locally symmetric metric.
 If $f$ denotes an $L^2$ eigenfunction of $\Delta_{(M,h)}$ with respect to some eigenvalue 
 $\lambda < ||\rho_{\bP}||^2$ and $p_0\in M$ some arbitrary point in $M$ then there is a 
 constant $C>0$ such that
 \begin{equation}
  | f(p)| \leq C\exp\left\{ \left(||\rho_{\bP}|| - \sqrt{||\rho_{\bP}||^2-\lambda}\right)d(p,p_0)\right\}
 \end{equation}
 for any $p\in M$.
\end{theorem}

\begin{remark}\label{remark}
It might be helpful to point out that the upper bounds in Theorem \ref{pointwise bound} also
follow from the theory of constant terms in automorphic forms. Briefly, an automorphic form is bounded by its constant term on cusp ends. The theory of automorphic forms also shows that the upper bounds
in Theorem \ref{pointwise bound} are sharp. 
Compact perturbations as in Theorem \ref{generalization} prevent the full application of the method of
automorphic forms, but the method of Fourier decompostion can still be applied to each cusp and might
be used to obtain similar upper bounds on eigenfunctions.

On the other hand, the method in this paper can be used to prove the same upper bound as in 
Theorem \ref{generalization} for perturbed metrics $g$ of the invariant metric $g_0$ of the
locally symmetric space $M=\Gamma\backslash X$ such that in each cusp end, 
when $y\to +\infty$, the difference
$g-g_0$ and its derivatives up to second order are of smaller order than $\e^{-2||\rho_{\bP}||y}$.
Briefly, under such perturbations, the continuous spectrum remains unchanged by applying the
decomposition principle to larger and larger compact subsets of $M$ and by observing the dependence
of the spectrum of the Laplace operator on the metric.

For such perturbed metrics on $M$, one can not use the method of Fourier decompostion on cusp
ends to study eigenfunctions.
\end{remark}

\section{$\boldsymbol{L^p}$ theory for eigenfunctions}\label{Lp theory for eigenfunctions}

Again, $M=\Gamma\backslash X$ denotes a locally symmetric space such that $X$
is a symmetric space of non-compact type and $\Gamma$ an arithmetic lattice with $\Q$-rank
one.\\

Let $f$ denote an $L^2$ eigenfunction of $\DM$ with respect to some eigenvalue $\lambda < ||\rho_{\bP}||^2$ below the continuous spectrum. 
From \cite[Proposition 3.3]{MR1016445} and the arguments in \cite{MR937635} it follows that
$f$ belongs to $L^p(M)$ whenever $\lambda$ lies outside the parabolic region 
$$P_p = \left\{ ||\rho_{\bP}||^2 - z^2 : z\in\C, |\re z|\leq ||\rho||\cdot |\frac{2}{p}-1|\right\}\subset \C.  $$
Of course, this result is only interesting if $p > 2$ as for $p\in [1,2]$ we always have
$L^p(M)\hookrightarrow L^2(M)$ since the volume of $M$ is finite. Therefore, we assume in the
following $p >2$.\\
The apex of the parabola $\partial P_p$ is at the point 
$z_p = ||\rho_{\bP}||^2- ||\rho||^2\cdot |\frac{2}{p}-1|^2\in \R$. Therefore, we have $\lambda < z_p$ and
hence $f\in L^p(M)$ if
$$ p < \frac{2}{ 1- \frac{1}{||\rho||}\sqrt{ ||\rho_{\bP}||^2-\lambda} }.$$
Our pointwise bound from Theorem \ref{pointwise bound} leads to a sharpening of this result:

\begin{theorem}\label{Lp eigenfunction}
 Let $M= \Gamma\backslash X$ denote a locally symmetric space with $\qrank(\Gamma)=1$ and
 $f$ an $L^2$ eigenfunction of $\DM$ with respect to some eigenvalue $\lambda < ||\rho_{\bP}||^2$.
 Then $f$ is contained in $L^p(M)$ if
 \begin{equation}\label{Lp estimate}
  p < \frac{2}{1-\sqrt{1-\frac{\lambda}{||\rho_{\bP}||^2}}}.
 \end{equation} 
\end{theorem}
\begin{proof}
It is enough to look at the $L^p$ norm of the eigenfunction $f$ within a cusp $\omega\times A_{\bP,t}$.
With respect to rational horocyclic coordinates we have (note that the volume form
of $X$ in these coordinates is given by $\exp(-2||\rho_{\bP}||y)dxdy$)
$$ \int_{\omega\times A_{\bP,t}} |f|^p dvol_X \leq 
C\int_t^{\infty} \exp\Big\{ p(||\rho_{\bP}|| - \sqrt{||\rho_{\bP}||^2-\lambda})y \Big\}\exp(-2||\rho_{\bP}||y)dy.$$
The last integral is finite whenever
$$  p(||\rho_{\bP}|| - \sqrt{||\rho_{\bP}||^2-\lambda}) < 2||\rho_{\bP}||$$
and this is equivalent to the claim.
\end{proof}
Hence, an $L^2$ eigenvalue $\lambda < ||\rho_{\bP}||^2$ is also an $L^p$ eigenvalue, i.e. an eigenvalue for the Laplace-Beltrami operator $\DMp$ on $L^p(M)$, if the inequality (\ref{Lp estimate}) holds true.\\

We finally remark, that Theorem \ref{Lp eigenfunction} could be obtained by the same arguments as in 
 \cite{MR937635,MR1016445} if one could prove that the $L^p$ spectrum $\sigma(\DMp)$
 coincides with $\{\lambda_0,\ldots,\lambda_r\}\cup P_p'$ as the apex of the 
 parabola $\partial P_p'$ is at the point $||\rho_{\bP}||^2- ||\rho_{\bP}||^2\cdot |\frac{2}{p}-1|^2\in \R$.


\bibliographystyle{amsplain}
\bibliography{dissertation}

\end{document}